\documentclass[preprint,12pt]{elsarticle}

\usepackage{amssymb}
 \usepackage{amsthm}

\usepackage{epstopdf}

\newtheorem{theorem}{Theorem}
\newtheorem{lemma}{Lemma}
\newtheorem{corollary}{Corollary}

\newtheorem{prop}{Proposition}

\newtheorem{problem}{Problem}

\begin{document}

\begin{frontmatter}



\title{Efficient proper embedding  of a daisy cube} %into a hypercube }


\author{Aleksander Vesel} 

\address{Faculty of Natural Sciences and Mathematics, University of Maribor,  SI-2000 Maribor, Slovenia }

\begin{abstract}
For a set $X$ of binary words of length $h$ the daisy cube $Q_h(X)$ is defined as the 
subgraph of the hypercube $Q_h$ induced by the set of all vertices on shortest paths that connect  vertices of $X$ with the vertex $0 ^h$. A vertex in the intersection of all of these paths is a minimal vertex of a daisy cube. A graph $G$ isomorphic to a daisy cube admits several isometric embeddings  into a hypercube. We show that an isometric embedding is proper if and only if the label 
$0 ^h$ is assigned to a minimal vertex of $G$. This result allows us to devise an algorithm
which finds a proper embedding of a graph isomorphic to a daisy cube into a hypercube in linear time.
\end{abstract}


\end{frontmatter}



\section{Introduction}

Hypercube is one of the most important interconnection scheme for multicomputers. 
An obstacle for a direct application of a hypercube is the fact that the 
number of different hypercubes is very small with respect to the wanted (maximum) number 
of nodes, that is to say,  the number of vertices of a hypercube is always equal to a power of two.
For that reason, several other  interconnection topologies for multicomputers based on hypercubes
have been  proposed. 
These classes of graphs have been devised with the intention of preserving most important 
properties of  a hypercube while allowing more variety of resulting specific graphs.    
The corresponding families of graphs are mostly various subgraphs of a hypercube, 
of which its isometric subgraphs, i.e.\ its induced subgraphs that preserve distances, are
of particular importance. A crucial problem in this scope is to find 
and embedding of a graph of this type to a hypercube (see for example \cite{epp, imkl-00, Win}). 

Quite recently, a new concept which led   to the class of graphs called daisy cubes has been proposed in \cite{daisy}.  It has been shown that daisy cubes are isometric subgraphs of 
a hypercube, moreover, they include several other important classes of graphs, e.g. Fibonacci and, 
Lucas cubes. It is known that the cube-complement of a daisy cube is also a daisy cube \cite{jaz}.
It is also proven that a class of graphs, which is of significant importance in chemical graph theory,  
also belongs to daisy cubes \cite{petra}.

In \cite{andrej}, daisy cubes are characterized  in terms of an expansion procedure. 
For a given graph $G$ isomorphic to a daisy cube, but without the corresponding
embedding into a hypercube, an algorithm
which finds a proper embedding of $G$ into a hypercube in $O(mn)$
time is also presented.  

Several  challenging open problems with respect to daisy cubes have been 
proposed  \cite{daisy, andrej}.  In this paper, we focus  our study to the following one.

\begin{problem}
Is there a faster way of finding the vertex $0^h$ of a daisy cube $Q_h(X)$ than the one
provided in \cite{andrej}?
\end{problem}

In \cite{andrej}, it is also noted that a positive answer to Problem 1 would give a linear time algorithm for finding a proper embedding  of a graph isomorphic to a daisy cube.

The paper is organized as follows.
In the next section some basic definitions, concepts and results needed in the sequel are given.
In Section 3, a notion of a minimal vertex of a daisy cube is introduced. 
Some necessary and sufficient conditions that a minimal vertex has to fulfill are also given.  
In Section 4, it is shown that an isometric embedding of a graph 
isomorphic to a daisy cube, but without the corresponding
embedding into a hypercube, can be constructed in linear time even if a minimal vertex of a daisy cube 
is unknown.
The last section shows that an isometric embedding devised in the Section 4 can be applied 
in order to find a proper embedding within the same time bound. 

\section{Preliminaries}

Let $B = \{0, 1\}$. If $b$ is a word of length $h$ over $B$, that is, 
$b = (b_1, \ldots, b_h) \in B^h$,
then we will briefly write $b$ as $b_1 \ldots b_h$.
If $x, y \in B^h$, 
then the {\em Hamming distance} $H(x,y)$ between $x$ and $y$ is the number 
of positions in which $x$ and $y$ differ. 

We will use $[n]$ for the set $\{1, 2, \ldots, n\}$.  %and $[k,n]$ for the set $\{k, k+1, \ldots, n\}$ in this paper.
%If $s$ is a string, then $|s|$ denotes the {\em length of $s$}. 

% If $|s| = n$, then we write $s=s_1s_2\ldots s_n$. 

The {\em  hypercube} of order $h$ or simply {\em  $h$-cube}, 
denoted by $Q_h$, is the graph $G=(V,E)$ where the vertex
set $V(G)$ is the set of all binary strings
$b=b_1 b_2 \ldots b_h$, $b_i \in \{0,1\}$ for all $i \in [h]$,
 and  two vertices $x,y \in V(G)$ are adjacent in $Q_h$ if and only if %$x$ and $y$ differ in precisely one place. 
the Hamming distance between $x$ and $y$ is equal to one.

For a binary string $b=b_1b_2\ldots b_n$, let $\overline b_i= 1 - b_i$ for  $i \in [h]$. 
The {\em weight} of $u \in B^h$ is 
$w(u) = \sum_{i=1}^h u_i$, in
other words, $w(u)$ is the number of 1s in the word $u$. 
For the concatenation of bits the power notation will be used, for instance 
$0^h = 0\ldots 0 \in  B ^h$.

If $G$ is a connected graph, then the distance $d_G(u, v)$ (or simply $d(u,v)$)  between vertices $u$ and $v$ is the length of a shortest $uv,v$-path 
(that is, a shortest path between $u$ and $v$) in $G$. 
The set of vertices lying on all shortest $u, v$-paths is called
the {\em interval} between $u$ and $v$ and denoted by $I_{G}(u, v)$ \cite{mulder3}.
 We will also write $I(u, v)$ when $G$ will be clear from the context.

If $G$ is a graph and $X \subseteq V(G)$, then $G[X]$ %(or simply $\langle 	 X \rangle $) 
denotes the subgraph of $G$ induced by $X$.

If $u$ is a vertex of a graph $G$, let $N(u)$ denote the set of neighbors of $u$. 
Moreover, let $N[u] = N(u) \cup \{  u \}$. 

Let $G=(V,E)$ be a graph. A mapping  $\alpha : V(G) \rightarrow V(Q_h)$
is an {\em isometric embedding} of $G$ into $Q_h$ if $d_{Q_h}(\alpha(u),\alpha(v))=d_G(u,v)$ 
for every $u,v\in V(G)$. 
 We will denote the $i$-th coordinate of $\alpha $ with $\alpha_{(i)}$, i.e. 
$\alpha = (\alpha_{(1)}, \alpha_{(2)}, \ldots, \alpha_{(h)})$.

Let $G$ be a connected graph.
The isometric dimension of $G$ is the smallest integer
$h$ such that $G$ admits an isometric embedding into $Q_h$. Isometric subgraphs of hypercubes are called {\em partial cubes}.

Let $\le$ be a partial order on $V(Q_h)$ defined with
$u_1 \ldots u_h \le v_1 \ldots v_h$ if $u_i \le v_i$
holds for all $i \in [h]$.  For $X \subseteq V(Q_h)$ the graph induced by the set 
$\{v \in V(Q_h) \, | \, v \le x$ for some $x \in X \}$ is a {\em daisy cube generated by $X$ 
or order $h$ } and denoted by $Q_h(X)$. 

Let also $\lor$, $\land$ and $\oplus$  denote the bitwise OR, bitwise AND and bitwise exclusive OR operator, respectively.

%If $u$ and $v$ are vertices of a graph $G$, then the {\em interval} $I_G(u, v)$ between $u$ and 
%$v$ in $G$ is the set of vertices lying on shortest $u, v$-paths, that is, $I_G(u, v) = \{z : d(u, v) =
%d(u,z)+d(z, v)\}$. 

By a slight abuse of definition, we will say that a graph $G$ is a daisy cube if it is isomorphic to a daisy cube generated by some $X \subseteq V(Q_h)$.
If  $G$ is a daisy cube $Q_h(X)$, then $G$ may admit
more than one isometric embedding of $G$ into the $h$-cube.
 Let $X_G \subseteq B^h$ be the set of labels of the vertices of 
 $G$ assigned by an isometric embedding $\alpha$, i.e. $X_G = \alpha(V(G))$. 
 We say that $\alpha$ is a {\em proper embedding} of $G$
 if $G$ is isomorphic to $Q_h(X_G)$. %Otherwise the embedding is improper.

%By a slight abuse of definition, we will say that a graph $G$ is a daisy cube if it is isomorphic to a daisy cube generated by some $X \subseteq V(Q_h)$.
%More formally, $G$ is a daisy cube if there exists an isometric embedding $\alpha$ of $G$ into $Q_h$ such that  
 %the image of $\alpha$ equals $X\subseteq V(Q_h)$  and  $Q_h(X)$ is a daisy cube generated by $X$. 
 Let $G$ be a graph isomorphic to a daisy cube  of order $h$ and let $\alpha$
 denote its  proper embedding. Note that  
 every permutation of indices of $\alpha$ yields basically the ``same'' embedding. 
 We say that proper   embeddings  $\alpha$ and $\beta$ are {\em equivalent} if 
$\beta$ can be obtained from $\alpha$ by a permutation of its indices.

%Note that if $x, y \in X$ and $y \le x$, then $Q_n(X) = Q_n(X \setminus \{y\})$. 
For a daisy cube $Q_h(X)$, let  $\widehat{X}$ denote the antichain consisting of 
the maximal elements of the poset $(X,\le )$. It was shown in \cite{daisy} that $Q_h(X)  = Q_h(\widehat{X})$.  
Hence, for a given set $X \subseteq B_n$ it is enough to consider the antichain $\widehat{X}$.
The vertices of $Q_h(X)$ from $\widehat{X}$ the are called the {\em maximal } vertices of $Q_h(X)$.
More generally, if $G$ is a daisy cube of dimension $h$ with a proper embedding $\alpha$ such that
for exactly one $v\in V(G)$ we have $\alpha (v)=0^h$, then 
$ X \subseteq V(G)$ is the set of {\em maximal vertices of $G$ with respect to $v$} 
if  $G \cong  Q_h(\alpha(X))$ and $\widehat{\alpha( X)}=\alpha(X)$. 
We also say that $v$ is the  {\em minimal  vertex  of $G$ with respect to $\alpha$}.
 
 The following result shows that a daisy cube is a subgraph of $Q_h$ induced 
 by the union of intervals between $0^h$ and the vertices from $\widehat{X}$ \cite{daisy}.

\begin{lemma}  \label{induced}
Let  $X \subseteq B^h$.  Then   
%$Q_h(X) \cong \langle 	\cup_{x\in \widehat X} I(0^n, x) \rangle 	$.
$Q_h(X) = Q_h[\cup_{x\in \widehat X} I(0^h, x)]	$.
\end{lemma}

\section{Minimal vertices of a daisy cube }

If $u \in V(Q_h(X))$, then $I(0^n, u)$  induces a  $w(u)$-cube in $Q_h(X)$. 
Note that if $x \in \widehat X$, then  the cube induced by  
$I(0^n, x)$  is maximal in $Q_h(X)$, i.e., 
it is not contained in any other cube that belongs to $Q_h(X)$. 

If $x \in B^h$, let $S^x$ denote the set of indices of $v$ with $x_i=1$, i.e., 
$S^x = \{ i \, | \, x_i=1, i \in [h] \}$.

Let $v \in  B^h$ and let $^v\!\beta :  B^h \rightarrow B^h$ be the function defined as 
$$^v\!\beta_{(i)}(u)  =
        \left \{\begin{array}{ll}
              u_i,  \; v_i = 0 \\
              \bar u_i,  \; v_i = 1 \\             
         \end{array} \right. $$

%The next result shows that a graph isomorphic to daisy cube may admit more than one minimal vertex - one minimal vertex for every proper embedding. 

\begin{lemma} \label{minimal}
%Let  $X \subseteq B^h$.  
Let  $G$ be a graph isomorphic to a daisy cube of order $h$ with a proper embedding $\alpha$ such that $\alpha(v^0)=0^h$ and $\widehat X \subseteq V(G)$  is its corresponding maximal set. 
%Let  $Q_h(X)$ be a daisy cube. 
%If $v \in \cap_{x\in \widehat X} I(0^n, x)$, then 
 If $v \in \cap_{x\in  \widehat X} I(v^0, x)$, then 
 
 (i) $^v\!\beta$ restricted to $\alpha(V(G))$ is a bijection that maps  to $\alpha(V(G)) $,
 
%  (ii)  $Q_h(X) \cong Q_h(\alpha^v( \widehat X ))$.

(ii) $^v\!\beta \circ \alpha$ is a proper embedding of $G$ with the minimal vertex $v$ and 
the  maximal vertex set
$Y =\{ y \, | \,  ^v\!\beta (\alpha(y)) = \alpha(x); \, x \in \widehat X \}$.
\end{lemma}

\begin{proof}
(i) 
 %$Q_h(\alpha(\widehat X))$.
We have to show that if $v \in \cap_{x\in  \widehat X} I(v^0, x)$, then 
for every $u  \in \alpha(V(G)) $) there is exactly one $ ^v\!\beta( u)  \in   \alpha(V(G))$. 
Note that $\alpha^{-1}(u)  \in  I(v^0, x)$ and $v  \in  I(v^0, x)$ for some $x \in  \widehat X$. 
Thus,  $S^{u} \subseteq S^{\alpha(x)}$ and $S^{\alpha(v)} \subseteq S^{\alpha(x)}$. 
It follows  that $S^{^v\!\beta (u)} \subseteq S^{\alpha(x)}$.
Since $\alpha$ is proper, $\alpha(V(G)) = \cup_{x\in \widehat X} I(0^h, \alpha(x))$ by Lemma \ref{induced} and  we obtain  $^v\!\beta (u) \in  V(\alpha( G))$.
%Thus, $^v\!\beta$ is surjective. 

In order to see that $^v\!\beta$ is injective, note that  $^v\!\beta(^v\!\beta(u))=u$ for every 
$u \in   \alpha(V(G))$. Suppose to the contrary that there exist $u,z  \in   \alpha(V(G))$, 
$u \not = z$, such that $^v\!\beta(u) = ^v\!\beta(z)$. It follows that 
$^v\!\beta(^v\!\beta(u)) = ^v\!\!\beta(^v\!\beta(z))$ and thus $u=z$, which yields a contradiction.

(ii) By (i), $^v\!\beta$ maps from $\alpha(V(G))$ to $\alpha(V(G))$. 
Let $x \in \widehat X$ and recall that $^v\!\beta(^v\!\beta(\alpha(x)))=\alpha(x)$. 
Thus,   if $y \in V(G)$ such that $\alpha(y) = \, ^v\!\beta(\alpha(x))$, we have 
 $^v\!\beta (\alpha(y)) = \alpha(x)$.
Moreover,  $^v\beta(v)=0^h$.
It follows that $Y =\{ y \, | \,  ^v\!\beta (\alpha(y)) = \alpha(x); \, x \in \widehat X \}$ is the 
maximal vertex set of $G$ with respect to
$^v\!\beta \circ \alpha$, while  $v$ is the corresponding minimal vertex. 
\end{proof}

%SI56 0317 6700 0217 783

\begin{figure}[!ht] 
	\centering
		\includegraphics[width=13.5cm]{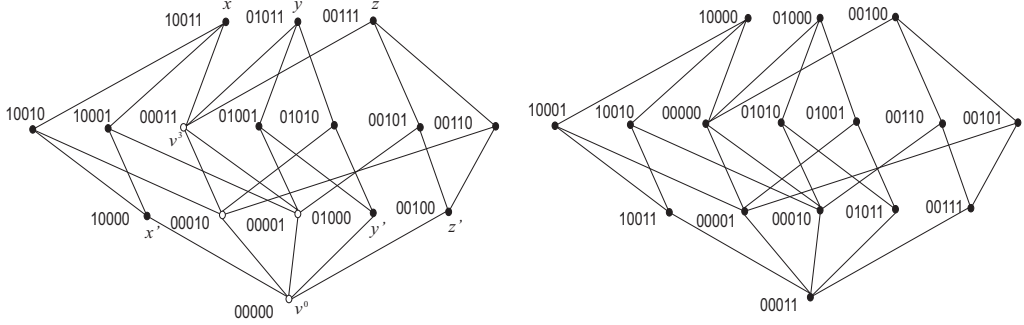}
        \caption{Two proper embeddings of a daisy cube.} 
\label{prva}
\end{figure}

Fig. \ref{prva} shows two proper embeddings of a daisy cube $G$. The embedding on the left hand side, 
say $\alpha$, admits the set of maximal vertices $\widehat X = \{x,y,z\}$  with labels 
$\alpha(x)= 10011$, $\alpha(y)= 01011$ and $\alpha(z)= 00111$. 
 Let $v^0 \in V(G)$ such that $v^0 = \alpha^{-1}(00000)$. Then $I(v^0,x) \cap I(v^0,y) \cap I(v^0,z) = 
 \{v^0, v^1, v^2, v^3 \}$, where %$\alpha(v^1)= 00001$, $\alpha(v^2)= 00010$ and 
 $\alpha(v^3)= 00011$. The embedding on the right  hand side of Fig. \ref{prva} is   
 $^{v^3}\!\beta \circ \alpha$ with the set of  
 maximal vertices $Y = \{x', y', z'\}$,  where the corresponding labels  are
$\alpha(x')= 10000$, $\alpha(y')= 01000$ and $\alpha(z')= 00100$. Note also that 
$^{v^3}\!\beta(\alpha(x'))= 10011$, $^{v^3}\!\beta(\alpha(y'))= 01011$ and 
$^{v^3}\!\beta(\alpha(z'))= 00111$.

If $v$ is a vertex of a partial cube $G$, then $N_G^v(u)$ (or simply $N^v(u)$ ) 
is the set of neighbors of $u$ which are closer to $v$ than $u$, more formally 
$N_G^v(u):= \{ z  \, |  \,  z \in N(u), d(v,z) =  d(v,u) - 1 \}$,

Let $u \in V(G)$ where $G = Q_h(X)$ and let  $X'$ be the maximal subset of $\widehat X$ with 
the property $u \in \cap_{x\in X'} I(0^h, x)$.
Let  $G^u$ be the graph induced by the set $\cup_{x\in X'} I(0^h, x)$, i.e. 
 $G^u = G[\cup_{x\in X'} I(0^h, x)]$. 
 Note that by Lemma \ref{minimal} and Lemma \ref{induced}, 
 $G^u$ is a daisy cube of order $h$ and $u$ is its minimal vertex.
Observe for example the graph 
$Q_4(0111, 1011, 1101, 1110)$ on the right hand side of Fig. \ref{druga}: if $u=1100$, 
then $X' = \{ 1110, 1101 \}$.

As noted in \cite{andrej}, an efficient way of finding a minimal vertex of a daisy cube $G$ 
would give a linear time algorithm for finding a proper embedding of $G$.
It was also shown that  if $G$ is a daisy cube of order $h$, then a minimal vertex 
of $G$ is of degree $h$. It is not difficult to see that a vertex of degree $h$ need not to be a minimal vertex of $G$. Note for example that $Q_h^-$ 
(that is a vertex deleted $Q_h$) admits $2^h-h-1$ vertices of degree $h$ and  exactly one 
minimal vertex (see also Fig. \ref{druga}, where $Q_4^-$ is depicted). 

\begin{prop} \label{x'}
Let $u \in V(G)$, where $G = Q_h(X)$ and $d(u)=h$. 
If $G^u = G[\cup_{x\in X'} I(0^h, x)]$, then every minimal vertex of $G$ belongs to 
 $\cap_{x\in X'} I(0^h, x)$. 
\end{prop} 

\begin{proof} 
Let $v$ be a minimal vertex of  $G$. Note that for every $x \in \widehat X$ and every 
$u \in I(0^h ,x)$ we have  $d(v,u) \le |S^x|$.
Suppose to the contrary that $v \not \in  \cap_{x\in X'} I(0^h, x)$.   
It follows that there exists $x \in X'$ such that $v \not \in I(0^h, x)$.  
If $u \in I(0^h, x)$, then $S^u \subseteq S^x$. Moreover, since $v \not \in I(0^h, x)$, there
exists an index $j \not \in S^x$ such that $v_j = 1$. It follows that the string $u$ defined by
$$ u_i  =
       \left \{\begin{array}{ll}
            \bar v_i,  \; i \in S^x  \\
             0, \; {\rm otherwise}
           \end{array} \right. $$
is a vertex of $I(0^h, x)$ with $d(v,u) > |S^x|$ and we obtain a contradiction.
 \end{proof}

\begin{theorem} \label{min}
If $G = Q_h(X)$ and $\hat x = \land_{x \in \widehat X} x$, then $v$ is a minimal vertex of $G$ if and only if $v \in \cap_{x\in  \widehat X} I(0^h, x) = I(0^h, \hat x)$.
\end{theorem} 

\begin{proof}
By Lemma \ref{minimal} and Proposition \ref{x'}, $v$ is a minimal vertex of $G$, 
 if and only if $v \in \cap_{x\in  \widehat X} I(0^h, x)$.
Note that $v \in \cap_{x\in  \widehat X} I(0^h, x)$ if and only if $S^v \subseteq   \cap_{x\in  X} S^x$. 
Since $S^{\hat x} = \cap_{x\in  X} S^x$, for every $v \in V(G)$ we have 
 $v \in \cap_{x\in  \widehat X} I(0^h, x)$  if and only if $v \le  \hat x$.
It follows  that $\cap_{x\in  X} I(0^h, x) = I(0^h, \hat x)$ and the assertion follows.
\end{proof}

¸%Given a vertex $v_0$ of a connected graph $G$, a spanning tree $T$ for which $d_T(v_0, v) = d_G(v_0, v)$ for every vertex $v \in V(G)$, is called a {\em BFS tree}. 

\section{Isometric embedding}

If  $G$ is a graph isomorphic to a hypercube (but without an embedding), then its 
isometric embedding is easy to obtain as shown in the next result.  
%(see for example \cite{KlIm}).
%\begin{prop} 
%Let  $G$ be a graph isomorphic to a $h$-cube, $v$ an arbitrary vertex of $G$ and
%$\alpha : V(G) \rightarrow V(Q_h)$ a function such that $\alpha(v)  =  0^d$, 
%the vertices of $N(v)$ obtain pairwise different labels of the form $0^{i-1}10^{h-i}$, $i\in [h]$,
%while for any other vertex $u \in V(G)$, the $i$-th coordinate of $\alpha(u)$ is equal to 1 if and only if a shortest $v,u$-path contains the vertex labeled  $0^{i-1}10^{h-i}$.
%Then $\alpha$ is  an isometric embedding of $G$ into $Q_h$. 
%Moreover, when a fixed embedding of $v$ and $N(v)$is chosen, $\alpha$ is unique. 
%\end{prop} 
%, otherwise the $i$-th coordinate of $\alpha(u)$ is set to 0. 

\begin{prop} \label{cubes}
Let  $G$ be a graph isomorphic to a $h$-cube, $v$ an arbitrary vertex of $G$ and
 $\alpha : V(G) \rightarrow V(Q_h)$   a function such that
$\alpha(v)  =  0^d$, 
the vertices of $N(v)$ obtain pairwise different labels of the form $0^{i-1}10^{h-i}$, $i\in [h]$,
while for any other vertex $u \in V(G)$ we set $\alpha(u) = 
 \lor_{z \in N^v(u)} \alpha(z)$. Then $\alpha$ is  an isometric embedding of $G$ into $Q_h$. 
 Moreover, when a labeling of vertices in $N[v]$ is chosen, $\alpha$ is unique.
\end{prop} 
\begin{proof} 
Since a hypercube is vertex-transitive, we may choose an arbitrary vertex $v$ of $G$ 
and set $\alpha(v)=0^h$. 
Moreover, for every $u \in V(G)$ with $d(v,u) = i$, $i \ge 1$, we must have 
$N^v(u) = \{ z \; | \; \overline{ \alpha_{(i)}(z) }=  \alpha_{(i)}(u) = 1$ 
for exactly one $i \in [h]$ and 
$\alpha_{(j)}(z) =  \alpha_{(j)}(u)$ for every $j \in [h] \setminus \{ i \} \}$. Thus, 
$\alpha(u) =  \lor_{z \in N^v(u)} \alpha(z)$. Thus, for chosen labeling of  
vertices in $N[v]$,  $\alpha$ is unique.
\end{proof} 

%\begin{lemma} 
%Let  $G$ be a graph isomorphic to a $d$-cube, $v$ an arbitrary vertex of $G$ and
% $\alpha$  the function $\alpha : V(G) \rightarrow V(Q_h)$ defined with 
%$$\alpha(u)  =
   %     \left \{\begin{array}{ll}
      %        0^d,  \; u = v \\
         %     0^i10^{d-i},  \; u \in N(v), \, {\rm where} \, i\in [d] {\rm \, defines \, an \, arbitrary \, ordering \, of \,} N(v)
  %\\
     %         \biglor_{z \in N_v(u)} \alpha(z), \; {\rm otherwise}
        %     \end{array} \right. $$
%It follows that $\alpha$   in an isometric embedding of $G$ into $Q_d$.
%\end{lemma}

%\begin{proof} 
%The proof is by induction on $d(v,u)$. If $d(v,u) \le 1$, the lemma clearly holds. Let 
%$d(v,u) = t > 1$ and assume the lemma holds for all vertices $z$ with  $d(v,z) < t$. 
%Note that the $i$-th coordinate of $\alpha(u)$  is equal to 1 if and only if there exists 
%$z \in N^v(u)$ such that the $i$-th coordinate of $\alpha(z)$  is equal to 1.
%Since  a shortest $v,u$-path contains $z$ and, 
%by induction hypothesis,  a shortest $v,z$-path contains  $0^{i-1}10^{h-i}$, it follows 
%that a shortest $v,u$-path contains $0^{i-1}10^{h-i}$. 
%By Proposition \ref{cubes}, this assertion concludes the proof.
%\end{proof} 

\begin{lemma} \label{partial}
Let  $G$ be partial cube of isometric dimension $h$, $u$ a vertex of degree $h$ in $G$ and
let for every $v \in V(G) \setminus N[u]$ it holds that $|N^u(v)|\ge 2$. Then we may define  
a function $\alpha : V(G) \rightarrow V(Q_h)$ such that  
$\alpha(u)  =  0^h$, 
the vertices of $N(u)$ obtain pairwise different labels of the form $0^{i-1}10^{h-i}$, $i\in [h]$,
while for any other vertex $v \in V(G)$ we set $\alpha(v) = 
 \lor_{z \in N^u(v)} \alpha(z)$.
 Moreover, 
 
 (i) $\alpha$ is  an isometric embedding of $G$ into $Q_h$,
 
 (ii) when a fixed embedding of vertices in $N[v]$ is chosen, $\alpha$ is unique. 
\end{lemma} 

\begin{proof} 
 Since $G$ is a partial cube of dimension $h$, we may assume that $G$ is an isometric subgraph of
  an (unlabeled) $h$-cube $H$. %, i.e. $V(G) \subseteq V(H)$ and $E(G) \subseteq E(H)$. 
 Let  $\beta$ be a embedding of $H$ with respect to $v$ as defined in Proposition 
 \ref{cubes} and let 
 $\alpha$ be an embedding of $G$ such that  for every $z \in N[u]$ we set $\alpha(z)=\beta(z)$.
 Since $|N_G^u(v)|\ge 2$ and $N_G^u(v) \subseteq N_H^u(v)$ for 
 every $v \in V(G) \setminus N[u]$, it follows that 
 $\alpha(v)=\beta(v)$ for every vertex $v \in V(G)$.  By Proposition \ref{cubes}, 
 $\beta$ is  an isometric embedding of $H$ into $Q_h$. Thus, 
  $\alpha$ is an isometric embedding of $H$ into $Q_h$. Moreover, 
  by Proposition \ref{cubes}, $\alpha$ is unique for a fixed embedding of vertices in $N[v]$. 
\end{proof} 

\begin{corollary} \label{posledica}
Let  $G$ be a graph isomorphic to a daisy cube of order $h$. 
If $v$ is a minimal vertex of $G$ and $\alpha$ an isometric embedding with $\alpha(v)=0^h$,
then  $\alpha$ is proper.
\end{corollary} 

\begin{proof}
Since $v$ is a minimal vertex of $G$, there exist a proper embedding, say 
$\beta$, such that $\beta(v)=0^h$.  We may also assume w.l.o.g. that for every $u \in N(v)$ we have  $\beta(u) = \alpha(u)$. From Lemma \ref{partial} then it follows
that $\beta(u) = \alpha(u)$  for every $v \in V(G)$.
\end{proof}

\noindent
{\bf Remark}\\
%However, 
If $G$ is isomorphic to a daisy cube and
$\alpha$ a proper embedding of $G$, then
different selections of labels for vertices of $N(u)$ yield different but 
equivalent proper embeddings. 

If $G$ is a partial cube and $\alpha $ its isometric embedding to $Q_h$,
let $W_i(G)$  denote the set of vertices of $G$ with weight  $i$, i.e. 
 $W_i(G) = \{ v \, | \, w(\alpha(v))= i \}$.
 
We will also need the following result.

\begin{prop} \label{kocka}
If $G$ is a partial cube, $\alpha $ its isometric embedding to $Q_h$ and 
$v \in V(G)$ such that $w(\alpha(v))=i$, then $|N(v) \cap W_{i-1}(G)| \le i$.
\end{prop} 
 
\begin{proof}
Since $\alpha $ is isometric embedding of $G$ to $Q_h$, for every 
$v \in V(G)$ with $w(\alpha(v))=i$, we have $N_G(v) \subseteq N_{Q_h}(v)$.
Moreover, $|N(v) \cap W_{i-1}(Q_h)| = i$ and therefore $|N(v) \cap W_{i-1}(G)| \le i$.
\end{proof} 

\begin{prop} \label{edge}
Let  $G = Q_h(X)$, $x,y \in \widehat X$ and $x\not = y$. 
If $u \in I(0^h,x)$ and $v \in I(0^n,y)$ such that  $u,v \not \in I(0^n, x) \cap I(0^h,y)$  then 
$uv \not \in E(G)$.
 \end{prop} 

\begin{proof} 
Suppose to the contrary that there exist $u \in I(0^h,x)$ and $v \in I(0^h,y)$ such that  $u,v \not \in I(0^h,x) \cap I(0^h,y)$  and $d(u,v) =1$. 
Since $\widehat X$ is maximal, there exist at least two indices $i,j \in [h]$, such that $x_i \not = y_i$ and  $x_j \not =  y_j$ (otherwise we have either $x \le y$ or $y \le x$). 
Suppose w.l.o.g. $x_i=1$, $y_j=1$ and  $u_k=v_k$ for every $k\not = i,j$. 
If $u_i=0$ (resp. $v_j=0$), then $u  \in I(0^h,y)$ (resp. $v  \in I(0^h,x)$). 
It follows that  $u_i=v_j=1$.  But then $u=v$ and we obtain a contradiction.
\end{proof}

\begin{prop} \label{neighbor}
Let  $G = Q_h( X)$. If $u\in V(G)$ and $d(u)=h$,   then  $N(u) \subseteq G^u$.
 \end{prop} 

\begin{proof} 
Let $G^u = G[\cup_{x\in X'} I(0^h, x)]$.
Suppose to the contrary that there exists $v \in N(u)$ such that
 $v \not \in G[\cup_{x\in X'} I(0^h, x)]$. It follows that there exists $y \in \widehat X - X'$ such 
 that $v \in I(0^h,y)$. Since $u \in I(0^h,x)$ for some $x \in \widehat X$ and $x \not = y$, Proposition  \ref{edge} yields a contradiction.
\end{proof} 

\begin{prop} \label{degree}
Let $G =Q_h(X)$ and $u \in V(G)$.
If   $d(u)=h$ and  $G^u = G[\cup_{x\in X'} I(0^h, x)]$,   then $|\cup_{x\in X'} S^x|=h$.
 \end{prop} 

\begin{proof} 
Suppose $|\cup_{x\in X'} S^x|<h$. It follows that there exist $j\in [h]$ such that for all 
$v\in \cup_{x\in X'} I(0^h, x)$ we have $v_j=0$. 
%It follows that $v_j=0$. 
Since $d(u)=h$ , there exists $z\in N(u)$ such that $z_j=1$. 
It follows that $z \not \in \cup_{x\in X'} I(0^h, x)$. Thus, there exists $y \in \widehat X - X'$ such that $v \in I(0^h, y)$. Proposition  \ref{neighbor} yields a contradiction.
\end{proof} 

%\begin{corollary} \label{nic}
%Let $G \cong Q_d(\widehat X)$ and $u \in V(G)$ with $d(u)=d$. Then    $|X'|=1$ if and only    $G=\cong Q_d$.
%\end{corollary} 
%\begin{prop} \label{010} Let $G \cong Q_h(\widehat X)$ and $u \in V(G)$.
%If   $d(u)=h$ and  $G^u = G[\cup_{x\in X'} I(0^h, x)]$,   then $0^{i-1}10^{h-i} \in G^u$.\end{prop} 
%\begin{proof} 
%By Proposition \ref{degree}, for every $i\in [h]$ there exists $x\in X'$ such that $x_i=1$. 
%It follows that every for every   $i\in [h]$   there exists $x\in X'$ such that $0^{i-1}10^{h-i} \in I(0^h,x)$.
%\end{proof} 

\begin{lemma} \label{main}
Let $G = Q_h( X)$ and $u \in V(G)$ such that $d(u)=h$. % and $G^u = G[\cup_{x\in X'} I(0^h, x)]$.
Then $|N^u(v)|\ge 2$ for every $v \in V(G) \setminus N[u]$.
\end{lemma} 

\begin{proof} 
Let $G^u = G[\cup_{x\in X'} I(0^h, x)]$.
By Lemma \ref{minimal} and Lemma \ref{induced}, $G^u$ is  a daisy cube and
$u$ its minimal vertex. It follows that the lemma holds for every
 $v \in V(G^u)$.
Suppose then that $v \not \in \cup_{x\in X'} I(0^h, x)$. Thus, there exists  
$y \in  \widehat X - X'$, such that $v \in I(0^h, y)$. 
Note that $S^u \subseteq \cap_{x\in X'} S^x$.

Let $S^{u+}  = \{ i \, | \, u_i=1, v_i=0 \}$ 
and $S^{u-}  = \{ i \, | \, v_i=1, u_i=0 \}$.  

We first show that  $|S^{u-}|\ge 2$. Suppose to the contrary that $|S^{u-}|= 1$, i.e., there exist 
exactly one index  $i \in [h] \setminus S^{u+}$, such that  $v_i=1$ and $u_i=0$. 
Since $d(u)=h$, by Proposition \ref{degree}, there  exists $x \in X'$ such that $x_i=1$.
Note also that $S^u \subseteq S^x$ and since $x_i=1$, we have $S^v \subseteq S^x$.
 It follows that $v \le x$ and  we obtain a contradiction.

If $|S^{u+}|=0$, then vertices  of  $I(u,v)$ induces a  $|S^{u-}|$-cube in $G$. Thus,
$v$ admits $|S^{u-}|$ neighbors at distance $d(u,v)-1$ from $u$. 
Since $|S^{u-}|\ge 2$, the case is settled. 

If $|S^{u+}|>0$, we may find $i,j \in  S^{u-}$ such that $i \not = j$. 
Let $z$ and $z'$ be vertices obtained from $v$ by setting the $i$-th and $j$-th coordinate to zero, respectively. Obviously, $z,z' \in N^u(v)$. This assertion completes the proof. 
\end{proof}

\begin{figure}[!ht] 
	\centering
		\includegraphics[width=13.5cm]{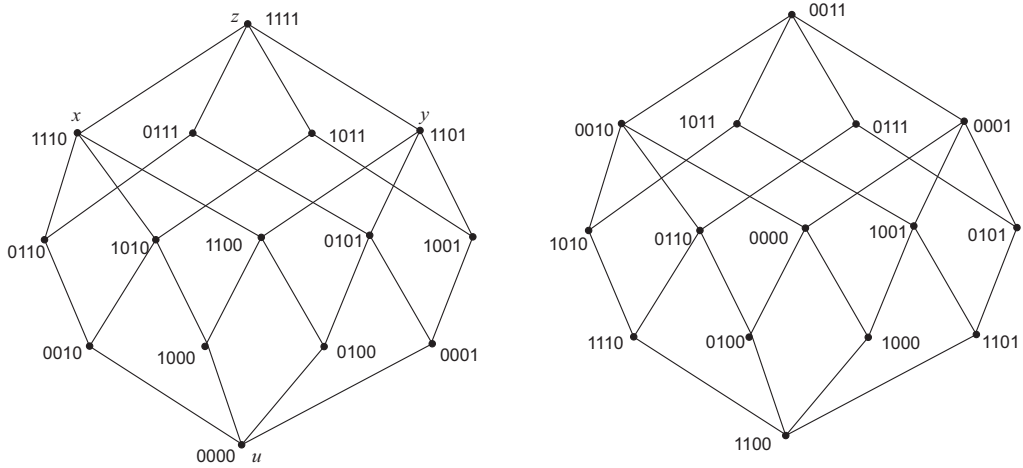}
        \caption{An isometric (left) and proper embeddings of a daisy cube isomorphic to $Q_4^-$.} 
        \label{druga}
\end{figure}

Lemma  \ref{main}  is the basis for the next algorithm which finds an isometric embedding for 
an unlabeled graph isomorphic to a daisy cube of dimension $h$.
%\newpage
\vskip .29cm
\noindent
{\bf Procedure} Embedding($G$, $h$, $\beta$, $u$); \\
%$\{$ $ab \in E(G)$ such that $d(a)=h$ and $d(b)=h-1$ $\}$ \\
\hphantom{ \ } 1. $u$ is a vertex of degree $h$ in $G$; \\
\hphantom{ \ } 2. $\beta(u) := 0^h$; \\
\hphantom{ \ } 3. $i := 0$; \\
\hphantom{ \ } 4. $Q := \emptyset;$  $\{ Q$ is an empty queue$\}$ \\
\hphantom{ \ } 5. {\bf for all}  $v \in V(G)$ {\bf do}  $p(v) := 0;$ \\
\hphantom{ \ } 6. {\bf for all}  $v \in N(u)$ {\bf do begin} \\
\hphantom{ \ \ \ \ \ \ \ \ }  $\beta(v):=0^{i-1}10^{h-i}$;   \\
\hphantom{ \ \ \ \ \ \ \ \ }  $i := i+ 1 $; \\
\hphantom{ \ \ \ \ \ \ \ \ }  $p(v) := u $; \\
\hphantom{ \ \ \ \ \ \ \ \ }  Insert $v$ in the end of $Q$; \\
\hphantom{ \ \ \ \ \  }  {\bf end}; \\
\hphantom{ \ } 7. {\bf while}  $Q \not = \emptyset$ {\bf do begin} \\
\hphantom{ \ \ \ \ \ \ \ \ } 7.1 Remove the first vertex $v$ from $Q$; \\
\hphantom{ \ \ \ \ \ \ \ \ } 7.2. {\bf for all}  $z \in N(v)$ {\bf do } \\
\hphantom{ \ \ \ \ \ \ \ \ \ \ \ \ \ \ \ \ \ }  {\bf if}  $p(z)=0$ {\bf then begin} \\
\hphantom{ \ \ \ \ \ \ \ \ \ \ \ \ \ \ \ \ \ \ \ \ }   $p(z) := v$;   \\
\hphantom{ \ \ \ \ \ \ \ \ \ \ \ \ \ \ \ \ \ \ \ \ }   Append $z$ to the end of $Q$; \\
\hphantom{ \ \ \ \ \ \ \ \ \ \ \ \ \ \ \ \ \ }  {\bf end} \\
\hphantom{ \ \ \ \ \ \ \ \ \ \ \ \ \ \ \ \ \ } {\bf else} $\beta(z)  := \beta(v) \lor \beta(p(z))$; \\
{\bf end.}

\begin{theorem} \label{a1}
If $G$ is a daisy cube, that the isometric embedding of $G$ can be found in linear time.
\end{theorem} 

\begin{proof} 
Note first that Lemma \ref{partial} defines the procedure to construct 
an isometric embedding of $G$ into $Q_h$. Let $\alpha$ and $\beta$ be isometric embeddings as defined in Lemma \ref{partial} and algorithm Embedding, respectively. 
Suppose that $u$ is the vertex being labeled $0 ^h$  both by the algorithm 
and by the construction of Lemma  \ref{partial}.  
Clearly, for every $v$ in $N[u]$ we could have $\alpha(v)=\beta(v)$.  
Note also that  in the essence the algorithm  performs a BFS search in $G$ (see for example \cite[Section 17.3]{imkl-00}).
Thus, for every $z \in N(v)$ of Step 7.2 we have $d(u,z) = d(u,p(z))+1 = d(u,v)+1$. 
It follows that  $v, p(z) \in N^u(z)$.  
By Lemma \ref{main}, since $d(u)=h$,  for every $v \in V(G) \setminus N[u]$ we have $|N_G^u(v)|\ge 2$. Therefore,  $\alpha(z)=\beta(z)$ for every 
$z \in  V(G) \setminus N[u]$. 

For the time complexity of the algorithm, note  that the number of the executions of the body of the loop
in Step 7.2 is bounded by the number of edges of a graph. Since the  time complexity 
of the body of the loop is constant, the overall number of step of the algorithm is linear in the number of the edges of the graph.
%If $G$ is a daisy cube of order $h$, then $G$ is an isometric subgraph of $Q_h$. 
%Since $u$ is of degree $h$ in $G$, we can always 
%find an isometric  embedding $\beta$ such that $\beta(u) = 0^h$ while the vertices of $N(u)$
%obtain the labels of the form $0^{i-1}10^{h-i}$. 
%Since $d(u)=h$, by Lemma \ref{main}, for every $v \in V(G) \setminus N[u]$ we have $|N_G^u(v)|\ge 2$. 
%Moreover, for every $v \in V(G) \setminus N[u]$  there exist at least two paths of length 
%$d_{Q_h}(u,v)$ in $G$. 
\end{proof} 

\section{Proper embedding}

\begin{lemma} \label{isotoprop}
Let $G$ be a daisy cube of order $h$, $v$ its minimal vertex and $u$ a vertex  of degree $h$ of $G$.
If $\beta$ is an isometric embedding of $G$ such that $\beta(u)=0^h$, then
%If $Y^u$ is the set of maximal vertices of $G^u$  with respect to $u$, 
$^v\!\beta \circ \beta$ is a proper embedding of $G$.
\end{lemma} 
\begin{proof}
Note that $^v\!\beta (\beta(v)) = 0^h$.
Since $\beta$ is isometric, it is easy to see that $^v\!\beta \circ \beta$  is also isometric.
Corollary \ref{posledica} now yields the assertion. 
%Thus, $^v\!\beta \circ \beta$  assign to the vertices of $N(v)$ 
%pairwise different labels of the form $0^{i-1}10^{h-i}$, $i\in [h]$.
%Moreover, by Lemma \ref{partial}, $G$ admits exactly one isometric embedding if the labels 
%of $N[v]$ are fixed. It follows that $^v\!\beta \circ \beta$ is a proper embedding of $G$.
\end{proof}

Let $u$ be a vertex of degree $h$ of $G = Q_h(X)$. % and $G^u = G[\cup_{x\in X'} I(0^h, x)]$.
Recall that  $G^u$ is a daisy cube of order $h$ and $u$ its minimal vertex.
If $\beta$ is an isometric embedding of $G$ such that $\beta(u)=0^h$,
let $Y^u$ be the set of maximal vertices of $G^u$ with respect to $u$ and let $Z^u$ be the set of vertices $z$ of 
$V(G) \setminus V(G^u)$ with the property $N^u(z) = N(z)$.

%By  for every $y \in Y$ we have $\beta(y) \in X'$

\begin{prop} \label{maxbeta}
Let $u$ be a vertex of degree $h$ of $G = Q_h(X)$ and $G^u = G[\cup_{x\in X'} I(0^h, x)]$.
If $\beta$ is an isometric embedding of $G$ such that $\beta(u)=0^h$,
%If $Y^u$ is the set of maximal vertices of $G^u$  with respect to $u$, 
then $Y^u =\{ y \, | \, \beta(y) = x, x \in X' \}$.
  \end{prop} 

\begin{proof} 
As noted above, $G^u$ is a daisy cube of order $h$ and
$u$ its minimal vertex. Since $u$ is of degree $h$ and  $\beta(u)=0^h$, 
the restriction of $\beta$ to $V(G^u)$ is a proper embedding of $G^u$.
Moreover, since every  permutation of indices of a proper embedding yields an  equivalent embedding, we may  
assume w.l.o.g. that for every $z \in N(u)$ we have $\beta(z) = 0^{i-1}10^{h-i}$
if and only if $u_i \not = z_i$. 
It follows that for every $w \in N(0^{h})$ we have  $^u\!\beta(\beta(w))=w$. By Lemma \ref{minimal}, $\beta \circ \beta$ is proper. Moreover, 
by Lemma \ref{partial},  $^u\!\beta(\beta(v))=v$ for every $v \in V(G^u)$.
From  Lemma \ref{minimal} then follows that $Y^u =\{ y \, | \, \beta(y) = x, x \in X' \}$.
\end{proof}

\begin{prop} \label{zbeta}
Let $u$ be a vertex of degree $h$ of $G = Q_h(X)$, $G^u = G[\cup_{x\in X'} I(0^h, x)]$ and
$z \in Z^u$. 
If   $\beta$ is an isometric embedding of $G$ and $\beta(u)=0^h$,
then there exists $y \in \widehat X - X'$ such that $z \in I(0^h, y)$. Moreover,
$$\beta_{(i)}(z)  =
        \left \{\begin{array}{ll}
               0,  \;  i  \in S^u \\
              y_i,  \; i \not \in S^u \\             
         \end{array} \right .$$
 \end{prop} 

\begin{proof} 
Let $G^u =G[\cup_{x\in X'} I(0^h, x)]$. By Lemma \ref{induced}, since 
$z \not \in  \cup_{x\in X'} I(0^h, x)$,  there must be $y
\in \widehat X - X'$ such that $z \in I(0^h, y)$. 
By $N^u(z) = N(z)$, we have $d(u,z) \ge d(u,v)$ for every $v \in  I(0^h, y)$.
If $v_i = 1$ for some $i \in S^u$,  then let $v'$ be the vertex of $G$ such that $v_j' = v_j$ for  
every $j\not = i$ and $v_i' = 0$. Obviously, $v' \le y$, thus $v' \in  I(0^h, y)$. Moreover, 
 since $\beta_{(i)}(v') = 1$, we have $d(u,v') > d(u,v)$ and we obtain a contradiction. 
 It follows that the assertion holds for 
 every  $i \in S^u$.  
 If $i \not \in S^u$, then $\beta_{(i)}(v) = v_i$ for every $v \in  I(0^h, y)$. 
 Since $y$ is maximal in $I(0^h, y)$, the assertion follows.
\end{proof} 

\begin{theorem} \label{minz}
Let $u$ be a vertex of degree $h$ of $G = Q_h(X)$. 
If   $\beta$ is an isometric embedding of $G$ such that $\beta(u)=0^h$,
$\hat y = \land_{y \in Y^u} \beta(y)$ and
$\hat z = \land_{z \in Z^u} \beta(z) \land 1^h$, then 
 $\beta^{-1}(\hat y \land \hat z)$ is a minimal vertex of $G$. 
 \end{theorem} 

\begin{proof} 
Note first that $\beta =  \beta^{-1}$, thus, for every $b \in B^h$ and every $i \in [h]$ it holds 
\begin{equation}
  \beta_{(i)}(b)  = \beta^{-1}_{(i)}(b)  =
        \left \{\begin{array}{ll}
               \bar b_i,  \;  i  \in S^u \\ \label{observation}
               b_i,  \; i \not \in  S^u\\             
         \end{array} \right .
\end{equation}

Let $\hat x = \land_{x \in X'} x$. By Proposition \ref{maxbeta},
 we have $Y^u =\{ y \, | \, \beta(y) = x, x \in X' \}$. Thus, $\hat x = \hat y$.
Note that by Proposition \ref{x'}, every  minimal vertex of $G$ belongs to $I(0^h, \hat x)$. 

If $X' = \widehat X$, then $Z^u = \emptyset$ and we get 
$\beta^{-1}(\hat y \land \hat z) = \beta^{-1}(\hat y)=\beta^{-1}(\hat x)$.  
By (\ref{observation}), we have $\beta^{-1}(\hat x) \le x$. It follow that 
$\beta^{-1}(\hat x)  \in I(0^h, \hat x)$ and we are done.

Otherwise, let $z \in Z^u$ such that $z \in I(0^h, y)$ for some $y \in \widehat X - X'$. 
We have to show that  
 $\beta^{-1}(\hat x \land \beta(z))$ is a minimal vertex of
 $ \cup_{x \in X'} I(0^h, x) \cup I(0^h, y)$, i.e.  
 $S^{\beta^{-1}(\hat x \land \beta(z))} \subseteq S^{\hat x \land y}$.
 
 By Proposition \ref{zbeta}, we have
 $$\beta_{(i)}(z)  =
        \left \{\begin{array}{ll}
               0,  \;  i  \in S^u \\
              y_i,  \; i \not \in  S^u  \\             
         \end{array} \right .$$

Since $S^u \subseteq S^{\hat x}$, we have 
 $$ (\hat x \land \beta(z))_i  =
        \left \{\begin{array}{ll}
              y_i,  \; i \not \in S^{\hat x} \setminus S^u \\     
               0,  \;  \, {\rm otherwise}\\        
         \end{array} \right .$$

By (\ref{observation}), we have $\beta^{-1}_i(\hat x \land \beta(z))  = 0$ for every 
$i \in [h]\setminus S^{\hat x \land y}$. Since we can repeat the above discussion for every $z \in Z^u$, the proof is complete.
\end{proof} 

Fig. \ref{druga} shows two embeddings  of a daisy cube $G$ isomorphic to $Q_4^-$.
The embedding $\beta$ on the left hand side is determined  such that $\beta(u)=0000$ (note that
$d(u)=4$). Since  $u$ is not minimal in $G$, the embedding $\beta$ is isometric but not proper. 
From $X'=Y^u = \{ x, y \}$ and $Z^u = \{ z \}$ we get $\hat y = 1110 \land 1101 = 1100$, 
$\hat z = 1111$ and $ \hat y  \land \hat z  = 1100 \land 1111 = 1100$. 
Moreover, the minimal vertex of  $G$ is $v=\beta^{-1}(1100)$ and 
$^v\!\beta \circ \beta$ is the proper embedding of $G$ as described in Lemma \ref{isotoprop}. That is to say, we obtain the proper embedding 
of $G$ by assigning $\beta(w) \oplus 1100$  to every $w \in V(G)$.

Theorem \ref{minz}  is the basis for the next algorithm, which finds a proper embedding of 
a graph isomorphic to a daisy cube of dimension $h$.

%\newpage
\vskip .29cm
\noindent
{\bf Procedure} Proper($G$, $h$, $\alpha$); \\
%$\{$ $ab \in E(G)$ such that $d(a)=h$ and $d(b)=h-1$ $\}$ \\
\hphantom{ \ } 1. Embedding($G$, $h$, $\beta$, $u$); \\
%\hphantom{ \ } 2. $Q$ := a list with $u$ as its only element; \\
\hphantom{ \ } 2. {\bf for }  $i :=1$ {\bf to} $h+1$ {\bf do} $W_i := \emptyset$; \\
\hphantom{ \ } 3. {\bf for all}  $v \in V(G)$ {\bf do }  
$W_{w(\beta(v))} := W_{w(\beta(v))} \cup \{ v \}$;  \\
\hphantom{ \ } 4. {\bf for all}  $v \in V(G)$ {\bf do} $q(v) := 0$; \\
%\hphantom{ \ \ \ \ \ \ \ \ } 5.1. 
\hphantom{ \ } 5. {\bf for }  $i :=1$ {\bf to} $h$  {\bf do begin} \\
\hphantom{ \ \ \ \ \ } 5.1. {\bf for all}  $x \in W_i$ {\bf do }  \\
\hphantom{ \ \ \ \ \ \ \ \ \  \ \ }  5.1.1 {\bf if}  $\sum_{y \in N(x) \cap W_{i-1}} q(y) = i(i-1)$ {\bf then begin} \\
\hphantom{ \ \ \ \ \ \ \ \ \ \ \ \ \ \ \ \ \ \ \ \ \ }   $q(x) := i$;   \\
\hphantom{ \ \ \ \ \ \ \ \ \ \ \ \ \ \ \ \ \ \ \ \ \ }   {\bf for all}  $y \in N(x) \cap W_{i-1}$ {\bf do}  
$q(y) := 0$;   \\
\hphantom{ \ \ \ \ \ \ \ \ \ \ \ \ \ \ \ \ \ \ }  {\bf end}  \\
\hphantom{ \ \ \ \ \ \ \ \ \ \ \ \ }  5.1.2 {\bf else if}  $N(x) \cap W_{i+1} = \emptyset $ {\bf then} $q(x) := i$ \\
\hphantom{ \ } 6. $s := 1^h$; \\
\hphantom{ \ } 7. {\bf for all}  $v \in V(G)$ {\bf do}  \\
\hphantom{ \ \ \ \ \ \ \ \ } 7.1. {\bf if }  $q(v) \not =  0$ {\bf then}  $s := s \land \beta(v)$;  \\
\hphantom{ \ } 8. {\bf for all}  $v \in V(G)$ {\bf do}  $\alpha(v) := s \oplus \beta(v)$;\\
{\bf end.}

 \begin{theorem} \label{proper}
A proper embedding of an unlabeled graph isomorphic to a daisy cube can be found in linear time.
 \end{theorem} 

\begin{proof} 
We first show that the algorithm Proper finds a proper embedding of $G$. As shown in Theorem 
\ref{a1}, embedding  $\beta$ provided by the algorithm Embedding is isometric. 
With respect to Theorem \ref{minz} and Step 7, we have to show that if 
$q(v) \not = 0$, then either $v\in Y^u$ or   $v\in Z^u$.
% If $v \not \in Y^u \cup Z^u$. 
Clearly, in Step 3, all vertices at distance $i$ from $u$ are inserted in $W_i$, while in 
Step 4,  $q(v)$ is set to 0 for every $v \in V(G)$.
The value of  $q(v)$ is altered  either in Step 5.1.1 or in Step 5.1.2. 

%If the value is altered  in 5.1.2, then 

Let $w(x)=i$. We show that $q(x)=i$ in the $i$-th iteration of for loop if and only if either 
$I(u,x)$ induces an $i$-cube or $x \in Z^u$. 
Note that $I(u,x)$ induces an $i$-cube, if and only $|N(x) \cap W_{i-1}| = i$ %N^u(x)$ admits exactly $i$ vertices 
and for every $y \in N(x) \cap W_{i-1}$ the set  $I(u,y)$ induces a $(i-1)$-cube.
Moreover, if $x\in Y^u$, then $I(u,x)$ induces a maximal $i$-cube in $G^u$. 

In the first iteration of Step 5, for every vertex  of $W_1$ the value of  $q$ is set to 1.
In the next iteration, when a vertex $x$ of $W_2$ is considered, these values for two vertices of 
$W_1$, say $y$ and $y'$,  are set to zero if $\{ u, y, y', x \}$ induce a 2-cube. 
Thus,  for every $x,y \in W_1 \cup W_2$ we have 

- $q(y)=1$ if and only if $x\in N(u)$ and there is no vertex $y \in W_2$ such that 
 $I(u,y) \subseteq I(u,x)$  and $I(u,x)$ induces $Q_2$.
 
 - $q(x)=2$ if and only $I(u,x)$ induces $Q_2$.

Suppose now that for $i\ge 3$  and $y \in W_{i-1}$ it holds that  $q(y)=i-1$ if  and only if 
$I(u,y)$ induces a maximal cube in $G[W_1 \cup W_2  \ldots \cup  W_{i-1}]$ or $N^u(y)=N(y)$;
otherwise, $q(y)=0$.
Let $w(x)=i$. Note that $|N(x) \cap W_{i-1}| \le i$ by Proposition \ref{kocka}. 
Thus, the condition of the if statement in Step 5.1.1 is fulfilled if and only if for every 
$y \in N(x) \cap W_{i-1}$ we have $q(y)=i-1$, i.e.  for every
$y \in N(x) \cap W_{i-1}$ the set $I(u,y)$ induces an $(i-1)$-cube. 
If the condition of the if statement returns true, then $q(x)$ obtains the value $i$ while 
for every  $y \in N(x) \cap W_{i-1}$ the value of  $q(y)$ is set to 0. 
If the condition of the if statement returns false, then $q(x)$  is set to $i$ if and only if 
$N(x) \cap W_{i+1} = \emptyset $, i.e. $x \in Z^u$.  
Thus, we showed that in the $i$-th iteration of the for loop $q(x)=i$ if and only if  either 
$I(u,x)$ induces an $i$-cube or $x \in Z^u$. 
Since the claim holds for every $i$, we showed that if 
$q(v) \not = 0$, $v \in V(G)$, then either $v\in Y^u$ or   $v\in Z^u$. 
From Theorem \ref{minz}  then it follows that the string $s$ computed 
in Step 7 is equal to $\hat y \land \hat z$, where 
$\hat y = \land_{y \in Y^u} \beta(y)$ and
$\hat z = \land_{z \in Z^u} \beta(z)$. 
By Theorem \ref{minz}, $\beta^{-1}(s)=v$ is a minimal vertex of $G$ while the embedding $\alpha$ obtained in Step 8 is equal to  $^v\!\beta \circ \beta$. Moreover,  $\alpha$  is  proper by Lemma \ref{isotoprop}.

In order to consider the time complexity of the algorithm,  
note first that all  steps of the algorithm except Step 5 can be executed in $O(m)$ time, where $m$ is the  number of edges of $G$. 
For the time complexity of Step 5 it is convenient to store the weights of vertices in a vector, which allows that the weight of a vertex and therefore its  inclusion in a set $W_i$ can be   
determined in constant time. Thus, the time complexity of Steps 5.1.1 and 5.1.2 is linear 
in the number of edges incident with the vertex $x$. Since Step 5 is performed for every 
vertex of the graph, the total number of steps is bounded by the number of edges of $G$. 
This assertion concludes the proof.
\end{proof}

\section*{Acknowledgements}
This work was supported by the Slovenian Research Agency under the grants P1-0297, J1-9109 and J1-1693. 
  % I am indebted to  anonymous reviewers for their careful reading and helpful suggestions which improve and clarify the paper.





\end{document}